\documentclass[12pt, a4paper, leqno]{amsart}
\usepackage[utf8]{inputenc}
\usepackage{amsmath}
\usepackage{amsfonts}
\usepackage{amssymb}
\usepackage{times}
\usepackage[T1]{fontenc} 
\usepackage{url} 
\usepackage{color,esint}
\usepackage[dvipsnames]{xcolor}
\usepackage[colorlinks, allcolors=RedViolet,pdfstartview=,pdfpagemode=UseNone]{hyperref} 
\usepackage{graphicx} 
\usepackage{enumitem}
\usepackage{tabularx}
    \usepackage{mathtools}
\usepackage{tikz}
\usepackage{mathrsfs}
\usetikzlibrary{arrows}

\let\oldmarginpar\marginpar
\renewcommand\marginpar[1]{\-\oldmarginpar[\raggedleft\footnotesize #1]%
{\raggedright\footnotesize #1}}

\usepackage{amsthm}
\theoremstyle{plain}
\newtheorem{theorem}[equation]{Theorem}
\newtheorem{lemma}[equation]{Lemma}
\newtheorem{proposition}[equation]{Proposition}
\newtheorem{corollary}[equation]{Corollary}

\theoremstyle{definition}
\newtheorem{definition}[equation]{Definition}

\newtheorem{example}[equation]{Example}

\theoremstyle{remark}
\newtheorem{remark}[equation]{Remark}

\numberwithin{equation}{section}

\newcommand{\R}{\mathbb{R}}

\newcommand{\Rn}{\mathbb{R}^n}
\newcommand{\Rm}{\mathbb{R}^m}

\renewcommand{\phi}{\varphi}
\renewcommand{\epsilon}{\varepsilon}
\def\le{\leqslant}
\def\leq{\leqslant}
\def\ge{\geqslant}

\def\phi{\varphi}
\def\rho{\varrho}
\def\vartheta{\theta}

\newcommand{\Phiw}{\Phi_{\text{\rm w}}}

\newcommand{\Phis}{\Phi_{\text{\rm s}}}

\def\esssup{\operatornamewithlimits{ess\,sup}}

\def\essinf{\operatornamewithlimits{ess\,inf}}

\DeclareMathOperator{\divop}{div}
\renewcommand{\div}{\divop}

\def\conv{{\rm conv}}

\newcommand{\ainc}[1]{\hyperref[ainc]{{\normalfont(aInc){\ensuremath{_{#1}}}}}}
\newcommand{\adec}[1]{\hyperref[adec]{{\normalfont(aDec){\ensuremath{_{#1}}}}}}
\newcommand{\inc}[1]{\hyperref[inc]{{\normalfont(Inc){\ensuremath{_{#1}}}}}}
\newcommand{\dec}[1]{\hyperref[dec]{{\normalfont(Dec){\ensuremath{_{#1}}}}}}
\newcommand{\azero}{\hyperref[azero]{{\normalfont(A0)}}}
\newcommand{\aone}{\hyperref[aone]{{\normalfont(A1)}}}

\newcommand{\condM}{\hyperref[condM]{{\normalfont(M)}}}
\newcommand{\aonen}[1]{\hyperref[aonen]{{\normalfont(A1-\ensuremath{#1})}}}
\newcommand{\VAn}[1]{\hyperref[VA1n]{{\normalfont(VA1-\ensuremath{#1})}}}
\newcommand{\wVAn}[1]{\hyperref[wVA1n]{{\normalfont(wVA1-\ensuremath{#1})}}}
\newcommand{\condMn}[1]{\hyperref[condM]{{\normalfont(M-\ensuremath{#1})}}}

\date{\today}

\begin{document}

\title[A fundamental condition in anisotropic generalized Orlicz spaces]{A fundamental condition for harmonic analysis in anisotropic generalized Orlicz spaces}
\author{Peter A.\ Hästö}


\begin{abstract}
Anisotropic generalized Orlicz spaces have been investigated in many 
recent papers, but the basic assumptions are not as well understood as in the 
isotropic case. 
We study the greatest convex minorant of anisotropic $\Phi$-functions 
and prove the equivalence of two widely used conditions in the theory of 
generalized Orlicz spaces, usually called (A1) and (M). 
This provides a more natural and easily verifiable condition for use 
in the theory of anisotropic generalized Orlicz spaces for 
results such as Jensen's inequality which we obtain as a corollary.
\end{abstract}

\maketitle


\section{Introduction}

This paper deals with generalized Orlicz spaces, also known as Musielak--Orlicz spaces. 
This is a very active field recently \cite{BaaB22, BenHHK21, BorC_pp, DefM21, HadV_pp, HarH21, MaeMO21, PapRZ22, SkrV21}, 
boosted by work on the double phase problem 
by Baroni, Colombo and Mingione, e.g.\ \cite{BarCM18, ColM15a}. The generalized 
Orlicz case unifies the study of the 
double phase problem and the variable exponent growth, widely researched over the last 20 years \cite{DieHHR11}. Many studies deal with isotropic energies of the type 
\[
\int_\Omega \phi(x, |\nabla u|)\, dx
\]
but recently also the anisotropic case 
\[
\int_\Omega \Phi(x, \nabla u)\, dx
\]
has been considered e.g.\ in \cite{AhmCGY18, BulGS21, Chl18, ChlGSW21, HasO_pp, LiYZ21}. As 
an example of an anisotropic energy with non-standard growth 
we could take a double-phase functional where the $q$-phase is directional:
\[
\int_\Omega |\nabla u|^p + a(x) |\partial_{x_1} u|^q\, dx\,;
\]
here only variation in the $x_1$-direction makes a contribution to the energy 
in the $q$-phase $\{a>0\}$. 

Counter examples (see, e.g., \cite{BalDS20, DieHHR11}) show that more advanced results such as the boundedness of averaging operators or density of smooth functions require connecting 
$\Phi(x, \xi)$ for different values of $x$. To this end, we developed in the isotropic case the \aone{} 
condition \cite{HarH19, Has15} (see also \cite{MaeMOS13a}), which is essentially 
optimal for the boundedness of the maximal operator. 
In the anisotropic case Chlebicka, Gwiazda, Zatorska-Goldstein and co-authors \cite{AhmCGY18, BulGS21, Chl18, ChlGSW21, ChlGZ18, ChlGZ18b, ChlGZ19, GwiSZ18} have developed a theory based on their \condM{} condition. 
To state and compare these conditions, let us define 
\begin{equation}\label{eq:phiPlusMinus}
\Phi_{B}^-(\xi):=\essinf_{x\in B\cap \Omega}\Phi(x,\xi)
\qquad\text{and}\qquad
\Phi_{B}^+(\xi):=\esssup_{x\in B\cap \Omega}\Phi(x,\xi).
\end{equation}
In essence, the \aone{} conditions says that $\Phi_B^+$ can be bounded 
by $\Phi_B^-$ in small balls $B\subset\Rn$ in a quantitative way, whereas \condM{} 
say that it can be similarly bounded by the least convex minorant 
$(\Phi_B^-)^\conv$ of $\Phi_B^-$ (see Definition~\ref{def:conditions}). Obviously, the latter is a stronger condition, and 
it is also more difficult to verify, since the relationship between 
$(\Phi_B^-)^\conv$ and $\Phi_B^-$ may be complicated in the anisotropic case. 

In the isotropic case $\Phi_B^- \lesssim (\Phi_B^-)^\conv$ so  
\condM{} and \aone{} are equivalent. In the anisotropic case this inequality 
does not hold (see Example~\ref{ex:noEquivalence}), but we are nevertheless able to prove the equivalence of the conditions by a more careful analysis. 

\begin{theorem}\label{thm:equiv}
Let $\Phi:\Omega\times\Rm\to [0,\infty]$ be a strong $\Phi$-function. 
Then \aone{} and \condM{} are equivalent.
\end{theorem}

This result and the techniques 
introduced in this paper will allow for the development of a 
theory of anisotropic generalized Orlicz spaces with more natural assumptions. 
As an example we prove the following Jensen-type inequality. 

\begin{corollary}[Jensen-type inequality]\label{cor:jensen}
Let $\Phi$ satisfy \aone{} and $f\in L^\Phi_\mu(\Omega; \Rm)$. Then there exists $\beta>0$ such that 
\[
\Phi_B^+ \bigg( \beta \fint_B f\, d\mu\bigg) 
\le 
\fint_B \Phi(x, f)\, d\mu +1 
\]
when $\rho_\Phi(f)\le 1$ and $\mu(B)\le 1$. 
\end{corollary}

Although the extra assumption $\rho_\Phi(f)\le 1$ in the corollary may seem strange, it 
follows naturally for instance when dealing with local regularity and it is known that the 
anisotropic Jensen-inequality does not hold without restrictions.

Let us next define precisely the concepts we are using and characterize functions 
$\Phi$ for which the equivalence $\Phi^\conv \simeq \Phi$ holds (Corollary~\ref{cor:Mstar}). 
In Section~\ref{sect:conditions}, we define the conditions \aone{} and \condM{} and 
give preliminary remarks regarding the definitions. Finally, in Section~\ref{sect:main} 
we prove the main results mentioned above.


\section{Almost convexity and the greatest convex minorant}

I refer to the monographs \cite{HarH19} and \cite{ChlGSW21} for background 
on isotropic and anisotropic generalized Orlicz spaces, respectively. 
We consider functions $\Phi:\Omega\times \Rm\to [0,\infty]$; the capital letter $\Phi$ is used to 
highlight the distinction from the isotropic case $L^\phi$ in \cite{HarH19} where 
$\phi:\Omega\times [0,\infty)\to [0,\infty]$. 
The idea is to define 
\[
\rho_\Phi(v) := \int_\Omega \Phi(x, v)\, d\mu
\qquad\text{and}\qquad
\|v\|_\Phi:= \inf\{\lambda>0 \,|\, \rho_\Phi(\tfrac v\lambda) \le 1\}
\]
for a vector field $v\in L^1_\mu(\Omega; \Rm)$. The space $L^\Phi_\mu(\Omega; \Rm)$ 
is defined by the requirement $\|v\|_\Phi<\infty$. 
We use the equivalence relation $\Phi\simeq\Psi$ 
which means that there exists $\beta>0$ such that 
\[
\Phi(\beta \xi) \le \Psi(\xi) \le \Phi(\tfrac \xi\beta).
\]
Here and in the rest of the paper $\beta$ denotes a parameter which is given by 
by one or more conditions; if the conditions hold with different $\beta_k$, then 
we can use $\beta:=\min_k\beta_k$ for all the conditions so that we may just 
as well use only the one common $\beta$. 
Since the parameter $\lambda$ is inside $\Phi$ in the definition of $\|v\|_\Phi$, this is the natural 
way to compare functions $\Phi$ (cf.\ Example~\ref{eg:constInside}). 
To ensure that the integral in $\rho_\Phi$ makes sense and $\|\cdot\|_\Phi$ is a norm 
we require some conditions. 

\begin{definition}
Let $\Omega\subset \Rn$ be an open set. 
We say that $\Phi:\Omega\times \Rm\to [0,\infty]$ is a \emph{strong $\Phi$-function}, 
and write $\Phi\in \Phis(\Omega)$,  
if the following four conditions hold:
\begin{enumerate}
\item 
$x \mapsto \Phi(x, \xi)$ is measurable for every $\xi\in\Rm$.
\item 
$\displaystyle \Phi(x, 0) = \lim_{\xi \to 0} \Phi(x,\xi) =0$ and $\displaystyle \lim_{\xi \to \infty}\Phi(x,\xi)=\infty$ for a.e.\ $x\in \Omega$.
\item
$\xi\mapsto\Phi(x,\xi)$ is continuous in the topology of 
$[0,\infty]$ for a.e.\ $x\in \Omega$.
\item 
$\Phi$ is convex for a.e.\ $x\in\Omega$:
\[
\Phi(x, \alpha \xi + \alpha'\xi') 
\le 
\alpha \Phi(x, \xi) + \alpha' \Phi(x, \xi'), \quad \alpha, \alpha'\ge 0,\ \alpha+\alpha'=1. 
\]
\end{enumerate}
\end{definition}

With these conditions, $\|\cdot\|_\Phi$ is a norm. Note that continuity in $\xi$ follows from 
convexity if $\Phi$ is real-valued and (3) is only needed to ensure that 
$\Phi$ does not jump to $\infty$. Note also that this class of strong $\Phi$-functions is 
broader than that studied \cite{ChlGSW21} since we do not require that 
upper and lower bounds in terms of $N$-functions independent of $x$. For instance, 
this definition allows for $L^1$- and $L^\infty$-type growth. In \cite{HarH19} in the 
isotropic case we relaxed (3) and (4) further, and so used ``strong'' for this class, 
even though it is still less restrictive than $N$-functions.

For the study of $\Phi$-functions depending on the space-variable $x$, 
we use local approximations with the functions $\Phi_B^+$ and $\Phi_B^-$ from \eqref{eq:phiPlusMinus}
\cite{HarH19, Has15}. However, $\Phi_B^-$ need not 
be convex even if each $\Phi(x,\cdot)$ is (just think of $\min\{t,t^2\}$). 
In the isotropic case, $\phi_B^-$ nevertheless satisfies the following 
weaker variant of (4) above:
\begin{enumerate}
\item[(W4)]
$\Phi$ is \emph{almost convex} if there exists $\beta>0$ such that 
\[
\Phi\big(x, \beta(\alpha \xi + \alpha'\xi')\big) 
\le 
\alpha \Phi(x, \xi) + \alpha' \Phi(x, \xi'), 
\] 
for a.e.\ $x\in \Omega$ and $\alpha, \alpha'\ge 0$ with $\alpha+\alpha'=1$.
\end{enumerate}
Unfortunately, even this does not hold for $\Phi_B^-$ in the anisotropic case (see Example~\ref{ex:noEquivalence}).


The constant $\beta$ in the almost convexity condition (W4) 
should be inside the function since we do not assume 
doubling, or even finite, functions, as the following example illustrates. 
A constant outside is possible, but too restrictive.

\begin{example}\label{eg:constInside}
Let $\phi_\infty(t):=\infty \chi_{(1,\infty)}(t)$ be the function generating 
the space $L^\infty(\Omega)$. Define $\Phi(\xi) := \phi_\infty(\|\xi\|_{1/2} ) 
= \phi_\infty(|\xi_1| + 2\sqrt{|\xi_1 \xi_2|}+|\xi_2|)$ 
in $\R^2$. Consider  $\alpha=\frac12$ and the basis vectors $\xi=e_1$ and $\xi'=e_2$ in (W4). 
Then $\|\tfrac{e_1+e_2}2\|_{1/2}=2$ so $\Phi(\tfrac{e_1+e_2}2) = \infty$ 
and the inequality 
\[
\Phi(\tfrac{e_1+e_2}2) \le \tfrac L2 [\Phi(e_1) + \Phi(e_2)] 
\]
does not hold for any $L<\infty$. However, the almost convexity inequality (W4)
\[
\Phi(\beta \tfrac{e_1+e_2}2) \le \tfrac12 [\Phi(e_1) + \Phi(e_2)] = 0. 
\]
holds for $\beta\le \frac12$ since in this case $\Phi(\beta \tfrac{e_1+e_2}2) = 0$. 
\end{example}

If we choose $\xi'=0$ in the almost convexity condition (W4), then we obtain
\begin{equation}\label{ainc}
\tag*{{(aInc)\(_1\)}}
\Phi(x, \beta\alpha \xi) 
\le 
\alpha \Phi(x, \xi) 
\qquad\text{for any}\quad
\alpha\in [0,1].
\end{equation}
In the special case $\beta=1$, i.e.\ for a convex function, we have 
\begin{equation}\label{inc}
\tag*{{(Inc)\(_1\)}}
\Phi(x, \alpha \xi) 
\le 
\alpha \Phi(x, \xi) 
\qquad\text{for any}\quad
\alpha\in [0,1].
\end{equation}
These inequalities mean that the function $t\mapsto \frac{\Phi(x, t\xi)}t$ is almost 
increasing or increasing, hence the notation \ainc{1} and \inc{1}. 
In \cite{HarH19} we showed that these inequalities are useful substitutes for 
convexity in the isotropic case. 
In particular, it is easy to see that $\Phi_B^+$ and $\Phi_B^-$ satisfy 
\ainc{1} or \inc{1} if $\Phi$ does. For the anisotropic case the 
almost convexity is more appropriate since it also carries information 
about non-radial behavior.

%
%

Let us denote by $\Phi^\conv$ the greatest convex minorant of $\Phi$. 
This function is often denoted by $\Phi^{**}$, since it can be obtained 
by applying the conjugation operation $^*$ twice \cite[Corollary~2.1.42]{ChlGSW21}, but we will not 
use this fact here. We next show a connection between the 
greatest convex minorant and the almost convexity condition. 
The following is a version of Carath\'{e}odory's Theorem from convex analysis. 
Probably it is known, but a proof is included for completeness, since I could not 
find a reference. 

\begin{lemma}\label{lem:caratheodory}
Let $\Phi:\Rm\to [0,\infty]$. Then 
\[
\Phi^\conv(\xi) = \min\bigg\{\sum_{k=1}^{m+1} \alpha_k \Phi(\xi_k) \,\bigg|\, 
\sum_{k=1}^{m+1} \alpha_k \xi_k =\xi,\ \sum_{k=1}^{m+1} \alpha_k = 1,\, \alpha_k\ge 0 \bigg\}.
\]
\end{lemma}
\begin{proof}
Consider the epigraph of $\Phi$, 
\[
E:= \{ (\xi, t) \in \Rm\times \R \,|\, \Phi(\xi)\le t \} \subset \R^{m+1}.
\]
By Carath\'{e}odory's Theorem (see, e.g., \cite[Theorem~2.1.3]{NicP18}), 
every point in the convex hull of $E$ can be represented as a convex 
combination of at most $m+2$ points $\xi_k$ from $E$. Furthermore, we observe 
that if any of the points $\xi_k$ are from the interior of $E$, then the convex 
combination is also in the interior of the convex hull. Thus the points 
of the boundary, i.e.\ the graph of $\Phi^\conv$, are given as a convex 
combination of points in the boundary of $E$, i.e.\ on the graph of 
$\Phi$. Hence 
\[
\Phi^\conv(\xi) = \sum_{k=1}^{m+2} \alpha_k \Phi(\xi_k) 
\quad\text{for some}\quad 
\sum_{k=1}^{m+2} \alpha_k \xi_k =\xi,\ \sum_{k=1}^{m+2} \alpha_k = 1,\, \alpha_k\ge 0.
\]
This is the claim, except with one extra point $\xi_{m+2}$. 

However, $\xi$ lies in the convex hull of $\xi_1,\ldots, \xi_{m+2}\in \Rm$. Thus by 
Carath\'{e}odory's Theorem in $\Rm$, $\xi$ can be expressed as the convex combination of 
at most $m+1$ of the points $\xi_1$, \ldots, $\xi_{m+2}$. By re-labeling if necessary, 
we obtain  
\[
\sum_{k=1}^{m+1} \alpha_k' \xi_k =\xi,\ \sum_{k=1}^{m+1} \alpha_k' = 1,\, \alpha_k'\ge 0.
\]
Since $\big(\xi_k, \Phi(\xi_k)\big)_{k=1}^{m+2}$, all lie on the same hyper-plane for a boundary point, we also have 
\[
\Phi^\conv(\xi) = \sum_{k=1}^{m+1} \alpha_k' \Phi(\xi_k). \qedhere 
\]
\end{proof}

We can now show that $\Phi\simeq \Phi^\conv$ for almost convex functions.

\begin{corollary}\label{cor:Mstar}
The function $\Phi:\Rm\to[0, \infty]$ is almost convex
if and only if $\Phi\simeq \Phi^\conv$.
\end{corollary}
\begin{proof}
Clearly, $\Phi^\conv\le \Phi$ since $\Phi^\conv$ is defined as a minorant of $\Phi$. 
Let $2^i\ge m+1$ and set $\alpha_k:=0$ and $\xi_k:=0$ for $k>m+1$. By the almost convexity condition, 
\[
\Phi\bigg(\beta^i \sum_{k=1}^{2^i}\alpha_k \xi_k \bigg)
\le 
\alpha_{i,1} \Phi\bigg(\beta^{i-1} \sum_{k=1}^{2^{i-1}}\frac{\alpha_k}{\alpha_{i,1}} \xi_k \bigg) + 
\alpha_{i,2} \Phi\bigg(\beta^{i-1} \sum_{k=2^{i-1}}^{2^i}\frac{\alpha_k}{\alpha_{i,2}} \xi_k \bigg)
\]
where 
\[
\alpha_{i,1} := \sum_{k=1}^{2^{i-1}}\alpha_k
\qquad\text{and}\qquad
\alpha_{i,2} := \sum_{k=2^{i-1}+1}^{2^i}\alpha_k. 
\]
Iterating this $i$ times, we obtain that
\[
\Phi\bigg(\beta^i \sum_{k=1}^{2^i}\alpha_k \xi_k \bigg)
\le 
\sum_{k=1}^{2^i} \alpha_k \Phi( \xi_k)
\]
By Lemma~\ref{lem:caratheodory} and this inequality,
\[
\Phi^\conv(\xi) \ge \min\bigg\{\sum_{k=1}^{2^i} \alpha_k \Phi(\xi_k) \,\bigg|\, 
\sum_{k=1}^{2^i} \alpha_k \xi_k =\xi,\ \sum_{k=1}^{2^i} \alpha_k = 1,\, \alpha_k\ge 0 \bigg\}
\ge 
\Phi(\beta^i \xi). 
\]
Thus the almost convexity implies that $\Phi\simeq \Phi^\conv$. 

If, on the other hand, $\Phi\simeq \Phi^\conv$ with constant $\beta$, then we directly obtain
\[
\Phi\big(\beta(\alpha \xi + \alpha'\xi')\big) 
\le
\Phi^\conv(\alpha \xi + \alpha'\xi') 
\le 
\alpha \Phi^\conv(\xi) + \alpha' \Phi^\conv(\xi')
\le
\alpha \Phi(\xi) + \alpha' \Phi(\xi'). \qedhere
\]
\end{proof}


For almost convex functions we easily obtain a Jensen inequality with an extra constant.

\begin{corollary}[Jensen's inequality]\label{cor:jensenAconv}
Let $E\subset \Rm$ have positive, finite measure $\mu(E)$. 
If $\Phi\in C(E; [0, \infty])$ is almost convex, then 
there exists $\beta$ such that 
\[
\Phi\bigg(\beta \fint_E f\, d\mu\bigg)
\le 
\fint_E \Phi(f)\, d\mu
\]
for every $f\in L^\Phi_\mu(E; \Rm)$.
\end{corollary}
\begin{proof}
By Corollary~\ref{cor:Mstar} and Jensen's inequality for the convex function $\Phi^\conv$, 
\[
\Phi\bigg(\beta \fint_E f\, d\mu\bigg)
\le 
\Phi^\conv\bigg( \fint_E f\, d\mu\bigg)
\le 
\fint_E \Phi^\conv(f)\, d\mu
\le
\fint_E \Phi(f)\, d\mu. \qedhere
\]
\end{proof}

%


\section{Definition of and remarks on conditions}\label{sect:conditions}

The \ainc{1} and almost convexity (W4) conditions connect $\Phi(x, \xi)$ 
for different values of 
$\xi$ with $x$ fixed. However, more advanced results such as the density of smooth functions 
in Sobolev spaces require connecting 
$\Phi(x, \xi)$ for different values of $x$, cf.\ \cite{BorC_pp}. This is the purpose of the conditions
\aonen{\Psi} and \condMn{\Psi}, which generalize \aone{} and \condM{}.  

However, let us first start with the more elementary condition \azero{}: 
there exists $\beta>0$ such that
\begin{equation}\tag{A0}\label{azero}
\Phi(x, \beta\xi) \le 1 \le \Phi(x,\tfrac1\beta \xi)
\end{equation}
for all $\xi \in \Rm$ with $|\xi|=1$ and all $x\in \Omega$. 
Note that \azero{} is implicit in the assumption 
$m_1(|\xi|)\le \Phi(x,\xi)\le m_2(|\xi|)$ for $N$-functions $m_1$ and $m_2$ 
used in \cite{AhmCGY18, BulGS21, Chl18, ChlGSW21, ChlGZ18, ChlGZ18b, GwiSZ18}. 
This property is inherited 
by other versions of $\Phi$:

\begin{lemma}
If $\Phi\in \Phis(\Omega)$ satisfies \azero{}, then so do $\Phi_B^+$, $\Phi_B^-$
and $(\Phi_B^-)^\conv$.
\end{lemma}
\begin{proof}
Taking the supremum or infimum over $x\in \Omega$ in \azero{} of $\Phi$ gives \azero{} for 
$\Phi_B^+$ and $\Phi_B^-$. Since $(\Phi_B^-)^\conv\le\Phi_B^-$, the left inequality 
of \azero{} follows for $(\Phi_B^-)^\conv$. 
If $|\xi|\ge \tfrac1\beta$, then by \inc{1} and \azero{} we conclude that 
\[
\Phi(x, \xi) \ge \beta\,|\xi|\, \Phi(x, \tfrac1\beta \tfrac\xi{|\xi|}) \ge  \beta\,|\xi|. 
\]
Hence, for all $\xi\in \Rm$, $\Phi(x, \xi) \ge (\beta\,|\xi|-1)_+$ 
(since $(\beta\,|\xi|-1)_+=0$ when $|\xi|\le \tfrac1\beta$) and 
so $\Phi_B^-(\xi) \ge (\beta\,|\xi|-1)_+$. 
But the right-hand side is a convex function, so it follows that 
$(\Phi_B^-)^\conv(\xi) \ge (\beta\,|\xi|-1)_+$ since $(\Phi_B^-)^\conv$ is 
defined as the greatest convex minorant. Consequently, 
\[
(\Phi_B^-)^\conv(\tfrac2\beta \xi) \ge (\beta\,\tfrac2\beta-1)_+ = 1,
\]
when $\xi \in \Rm$ with $|\xi|=1$, so $(\Phi_B^-)^\conv$ satisfies \azero{} 
with constant $\tfrac\beta2$. 
\end{proof}

The condition \aone{} was introduced in \cite{Has15} (see also \cite{HarH19, MaeMOS13a}) and is essentially 
optimal for the boundedness of the maximal operator in isotropic generalized Orlicz spaces. 
It also implies the H\"older continuity of solutions and (quasi)minimizers 
\cite{BenHHK21, HarHL21, HarHT17}. For higher regularity, 
we introduced in \cite{HasO22} a vanishing-\aone{} condition along the 
same lines. These previous studies apply to 
the isotropic case, i.e.\ $m=1$. 
In \cite{HasO_pp, HasO_pp2} we generalized the \aone{}-conditions to the anisotropic case, 
although only the quasi-isotropic case was considered in the main results.
 
Chlebicka, Gwiazda, Zatorska-Goldstein and co-authors \cite{AhmCGY18, BulGS21, Chl18, ChlGSW21, ChlGZ18, ChlGZ18b, ChlGZ19, GwiSZ18} considered the assumption \condM{} in the anisotropic case; 
in the next definition their condition is reformulated to make it easier to 
compare with the \aone{} condition (see also Lemma~\ref{lem:azero}); also note that some of the earlier works included 
additional restrictions in the condition.
 
\begin{definition}\label{def:conditions}\label{aonen}
Let $\Phi, \Psi\in \Phis(\Omega)$. We say that $\Phi$ satisfies \aonen{\Psi} or \condMn{\Psi}
if for any $K>0$ there exists $\beta>0$ such that
\begin{equation}\tag{A1-$\Psi$}\label{aone}
\Phi_B^+(\beta \xi)\leq \Phi_B^-(\xi)+1
\qquad\text{when }\Psi_B^-(\xi)\le \tfrac K{\mu(B)}
\end{equation}
or
\begin{equation}\tag{M-$\Psi$}\label{condM}
\Phi_B^+(\beta \xi)\leq (\Phi_B^-)^\conv(\xi)+1
\qquad\text{when }(\Psi_B^-)^\conv(\xi)\le \tfrac K{\mu(B)}
\end{equation}
for all balls $B\subset \Rn$ with $\mu(B)\le 1$ and $\xi\in\Rn$.

When $\Psi(t):=t^s$ and $\Psi:=\Phi$ we use the abbreviations \aonen{s}, \aone{}, \condMn{s} and 
\condM{}. 
\end{definition}

The role of $\Psi$ is to calibrate the almost continuity requirement with the 
information on the function we are interested in and was developed from the initial 
condition \aone{} over the course of several studies \cite{BenHHK21, HarHT17, HarHL21}. 
For instance, we showed in \cite[Theorem~3.9]{BenHHK21} that the weak Harnack inequality holds for non-negative supersolutions of 
$\div(\phi'(|\nabla u|)\frac{\nabla u}{|\nabla u|})=0$ 
if the isotropic $\Phi$-function $\phi$ satisfies \aonen{\psi} and the supersolution 
satisfies $u\in W^{1,\psi}(\Omega)$, where $\psi\in\Phiw(\Omega)$ is a potentially different 
function. Note that this involves a trade-off, since 
larger $\psi$ means more restriction on $u$ and less restriction on $\phi$. 

As far as I know, Chlebicka, Gwiazda, Zatorska-Goldstein and co-authors considered 
\condM{} only in the case $\Psi(t):=t$ and $\Psi(t):=t^p$ (i.e.\ \condMn{1} and \condMn{p} 
in the notation above). However, the next example illustrates why this does not 
 lead to optimal results. 

\begin{example}[Variable exponent double phase]\label{ex:A1}
Let $\phi(x, t):= t^{p(x)} + a(x) t^{q(x)}$ where $a\in C^{0, \alpha}(\Omega)$, $a\ge0$ and 
$1<p \le q$. Now the \aone{} or \condM{} conditions reduce to 
\[
\frac{q(x)}{p(x)} \le 1+ \frac \alpha n
\quad\Longleftrightarrow\quad
\Big(\frac{q}{p}\Big)^+ \le 1+ \frac \alpha n
\]
Let $p^-:= \inf_{x\in\Omega} p(x)$ and 
$p^+:= \sup_{x\in\Omega} p(x)$. 
If we only use fixed exponent gauges such as \aonen{p^-} or \condMn{p^-}, then we 
instead end up with the condition 
\[
\frac{q(x)}{p^-} \le 1+ \frac \alpha n 
\quad\Longleftrightarrow\quad
\frac{q^+}{p^-} \le 1+ \frac \alpha n 
\]
which is worse, and quite unnatural as the largest value of $q$ is bounded by 
the smallest value of $p$. 
\end{example}

As a final remark about the formulation, we note that earlier papers used a 
form without the ``+1'' and instead restricted the range of $\Psi_B^-$. However, 
if \azero{} holds, then these formulations are equivalent. We prove it for \condM{}, 
the same applies to \aone{}.

\begin{lemma}\label{lem:azero}
Let $\Phi\in \Phis(\Omega)$ satisfy \azero{}. Then \condM{} holds if and only if 
\[
\Phi_B^+(\beta \xi)\leq (\Phi_B^-)^\conv(\xi)
\qquad\text{when }(\Phi_B^-)^\conv(\xi)\in [1, \tfrac K{\mu(B)}].
\]
\end{lemma}
\begin{proof}
If the condition of the lemma holds, then \condM{} needs only to be checked when 
$(\Phi_B^-)^\conv(\xi)\le 1$. This inequality and \azero{} imply that 
$|\xi|\le \frac1\beta$. Thus $\Phi_B^+(\beta^2 \xi)\le 1$ by \azero{}, so 
\condM{} holds with constant $\beta^2$. 

Assume conversely that \condM{} holds and $(\Phi_B^-)^\conv(\xi)\in [1, \tfrac K{\mu(B)}]$. 
Then it follows that 
\[
\Phi_B^+(\beta \xi)\leq (\Phi_B^-)^\conv(\xi) + 1 
\le
2(\Phi_B^-)^\conv(\xi).
\]
Then \inc{1} implies that $\Phi_B^+(\frac\beta2 \xi)\le \frac12\Phi_B^+(\beta \xi)\le (\Phi_B^-)^\conv(\xi)$, so the condition of the lemma holds with constant $\frac\beta2$. 
\end{proof}


\section{Equivalence of conditions}\label{sect:main}

In the previous section we introduced and motivated the conditions 
\aone{} and \condM{} and their variants. We now move on to the main result, and 
consider their relation to one another.

Since $(\Phi_{B}^-)^\conv \le \Phi_{B}^-$, \condMn{\Psi} implies \aonen{\Psi}. 
If $\phi$ is isotropic and satisfies \ainc{1}, then I showed in \cite{Has15} that 
$\phi_{B}^-(\beta t) \le (\phi_{B}^-)^\conv(t)$.
Hence the two conditions are equivalent in this case. However, as pointed out in 
\cite[Remarks 2.3.14 and 3.7.6]{ChlGSW21}, this approach is not possible in the anisotropic case. 
Since I did not understand the examples implicit in these remarks without consulting the 
authors, I include here an explicit example based on ideas of Piotr Nayar communicated 
to me by Iwona Chlebicka. 

\begin{example}\label{ex:noEquivalence}
Let $m=2$ and $\Phi_k((\xi_1, \xi_2)):= \xi_k^2$. Then both $\Phi_1$ and $\Phi_2$ are convex and 
$\Phi_1(e_2)=\Phi_2(e_1)=0$. Denote $\Phi:=\min\{\Phi_1,\Phi_2\}$. It follows that 
\[
\Phi^\conv(\alpha_1e_1+\alpha_2e_2) 
\le \alpha_1\Phi^\conv(e_1)+\alpha_2\Phi^\conv(e_2)
\le \alpha_1\Phi(e_1)+\alpha_2\Phi(e_2) = 0, 
\]
where $\alpha_1+\alpha_2=1$ and $\alpha_1,\alpha_2\ge 0$. Thus we see that $\Phi^\conv\equiv 0$. 
Since $\Phi(\beta(e_1+e_2)) = \Phi_1(\beta(e_1+e_2))=\beta^2$ but $\Phi^\conv\equiv 0$, 
the relation $\Phi\simeq \Phi^\conv$ does not hold. 
\end{example}

Even though $\Phi_{B}^-(\beta\xi) \le (\Phi_{B}^-)^\conv(\xi)$ does not hold in general, we next show that 
we can construct an almost convex minorant which is comparable to $\Phi_B^-$  when \aone{} holds, 
and can be used in \condM{}. We prove that \aone{} implies \condM{} in the 
main case $\Psi:=\Phi$, which corresponds to the natural energy space $L^\Phi$ or 
$W^{1,\Phi}$. The implication for \aonen{\Psi} and \condMn{\Psi} when $\Psi\ne \Phi$ remains an open problem.

Let $\Phi:\Rm\to [0,\infty]$ be a strong $\Phi$-function independent of $x$. 
Denote $K_s:= \{ \Phi \le s\}$ and 
observe that it is a convex compact set which includes $0$ in its interior. Define 
\[
\| \xi \|_{K_s} := \inf\{ \lambda>0 \,|\, \tfrac \xi\lambda \in K_s \}
\qquad
\text{and}
\qquad
N_s(\xi) := s \max\{1, \| \xi \|_{K_s} \}. 
\]
Here $\| \cdot \|_{K_s} $ is the Minkowski functional of the set $K_s$, first 
studied by Kolmogorov \cite{Kol34}. The Luxemburg norm $\|\cdot\|_\Phi$ defined previously is another 
example of a Minkowski functional. 
Note that $N_s$ is a convex function with $\{N_s \le s\} = K_s$. 
Since $\Phi$ is convex, $\Phi(\lambda \xi) \le \lambda \Phi(\xi)$ for $\lambda\le 1$. Thus 
$N_s \le \Phi$ outside $K_s$, $N_s\ge \Phi$ in $K_s$ and $N_s=\Phi$ on the boundary 
$\partial K_s$. In other words, we take the $s$-level set of $\Phi$ and replace 
$\Phi$ outside of it by the function $N_s$ which grows linearly. 


In Example~\ref{ex:noEquivalence} we showed that the minimum of two convex functions need 
not be even almost convex. However, in the next proposition we show that $\min \{\Phi, N_s\}$ 
is almost convex, since the two functions are somehow compatible. This will be used 
to construct a convex minorant of $\Phi_B^-$. The proposition also demonstrates the utility of 
the almost convexity condition, as it seems much more difficult to choose $N_s$ to make the 
minimum convex while still being a minorant of $\Phi_B^-$. 

\begin{proposition}\label{prop:almostConvex}
Let $\Phi:\Rm\to [0,\infty]$ be a strong $\Phi$-function.
Then $M_s := \min \{\Phi, N_s\}$ is almost convex.
\end{proposition}
\begin{proof}
Note that $M_s=\Phi\chi_{K_s} + N_s \chi_{\Rm\setminus K_s}$ and let $\alpha,\alpha'>0$ with 
$\alpha+\alpha'=1$. If $\xi,\xi' \not\in K_s$, then the convexity of $N_s$ implies that 
\[
M_s\big(\beta (\alpha \xi + \alpha'\xi')\big) 
\le
N_s\big(\beta (\alpha \xi + \alpha'\xi')\big) 
\le
\alpha N_s(\xi) + \alpha' N_s(\xi')
=
\alpha M_s(\xi) + \alpha' M_s(\xi').
\]
If $\xi,\xi'\in K_s$, then the inequality follows from the convexity of $\Phi$, 
which holds by assumption. Therefore it suffices to show that 
\[
M_s\big(\beta (\alpha \xi + \alpha' \xi')\big)
\le
\alpha \Phi(\xi) + \alpha' N_s(\xi') 
\]
when $\xi\in K_s$ and $\xi'\not\in K_s$. Define $\tilde \xi := \frac12(\alpha \xi + \alpha' \xi')$. 
We will show that 
\begin{equation}\label{eq:Cineq}
M_s(\tilde\xi)
\le
C(\alpha \Phi(\xi) + \alpha' N_s(\xi')).
\end{equation}
Observe that $M_s$ satisfies \inc{1}, since $N_s$ and $\Phi$ do.
By \inc{1}, \eqref{eq:Cineq} implies the previous inequality with constant 
$\beta:=\frac1{2C}$ and concludes 
the proof. We consider two cases to prove \eqref{eq:Cineq}.

\textbf{Case 1: $\tilde \xi \in K_s$.} Then $M_s(\tilde \xi) 
\le s \le N_s(\xi')$ and so \eqref{eq:Cineq} holds with $C=2$ when $\alpha' > \frac12$. 
Thus we may assume that $\alpha \ge \frac12$. Now if $M_s(\tilde \xi)\le 2 \Phi(\xi)$, then 
\eqref{eq:Cineq} holds with $C=4$. Hence we 
further assume that $M_s(\tilde \xi)> 2\Phi(\xi)$. 

We may assume that $\xi$, $\xi'$ and $0$ are not collinear since  
in the collinear case we can choose $\xi'_k\to \xi'$ such that $\xi$, $\xi'_k$ and $0$ are 
not collinear and use the continuity of $M_s$ and $N_s$. 
Let $\zeta'$ be the intersection of the segment $[0, \xi']$ and the line through $\xi$ and $\tilde \xi$ (see Figure~\ref{fig:construction}). 
If $\zeta'\in K_s$, then $\tilde \xi = \theta \xi + (1-\theta)\zeta'$ for some $\theta\in (0,1)$. 
By the convexity of $\Phi$ and $M_s(\tilde \xi)> 2\Phi(\xi)$ we have 
\[
M_s(\tilde \xi) = \Phi(\tilde \xi) 
\le \theta \Phi(\xi) + (1-\theta) \Phi(\zeta')
\le \tfrac12 M_s(\tilde \xi) + M_s(\zeta'). 
\]
Thus $M_s(\zeta')=\Phi(\zeta') \ge \frac12M_s(\tilde \xi)$. If, on the other hand, $\zeta'\not\in K_s$, then $M_s(\zeta')\ge s\ge M_s(\tilde \xi)$. 
In either case, we have $M_s(\zeta')\ge \frac12 M_s(\tilde \xi)$.

\begin{figure}[ht!]
\begin{tikzpicture}[scale=0.8]
\usetikzlibrary{calc}
\coordinate[label=left:$\xi$]  (A) at (-2,4);
\coordinate[label=above:$2\tilde \xi$] (C) at (1,4);
\coordinate[label=right:$\xi'$] (B) at (10,4);
\coordinate[label=below:$0$] (O) at (0,0);
\fill (A) circle[radius=2pt];
\fill (B) circle[radius=2pt];
\fill (C) circle[radius=2pt];
\fill (O) circle[radius=2pt];

\coordinate[label=below:$\,\tilde \xi$] (d) at ($ (O)!.48!(C) $);
\coordinate[label=below:$\zeta'$] (e) at ($ (O)!.19!(B) $);
\coordinate[label=above:$\eta'$] (e) at ($ (O)!.27!(B) $);
\fill ($ (C)!.5!(O) $) circle[radius=2pt];
\fill ($ (O)!.2!(B) $) circle[radius=2pt];
\fill ($ (O)!.25!(B) $) circle[radius=2pt];
\coordinate[label=right:$2\tilde \xi - \xi$] (E) at (3, 0);
\fill (E) circle[radius=2pt];

\draw [line width=1pt] (A) -- (B) -- (O) -- cycle;
\draw [line width=1pt] (C) -- (O);
\draw [line width=1pt, dashed] (A) -- (E);
\draw [line width=1pt, dashed] (C) -- (E);
\draw [line width=1pt, dashed] (O) -- (E); 
\end{tikzpicture}
\caption{Construction of auxiliary points}\label{fig:construction}
\end{figure} 

Consider the parallelogram $(0, \xi, 2\tilde\xi, 2\tilde\xi-\xi)$. Let $\eta'$ be the 
intersection of the segments $[2\tilde\xi, 2\tilde\xi-\xi]$ and 
$[0,\xi']$ (see Figure~\ref{fig:construction}). 
From $2\tilde\xi=(1-\alpha')\xi+\alpha'\xi'$ we observe that 
\[
\alpha'  
=
\frac {|2\tilde \xi-\xi|}{|\xi-\xi'|}
=
\frac{|\eta'|}{|\xi'|}
\ge
\frac{|\zeta'|}{|\xi'|} =: \nu \in (0,1);
\]
the second equality follows since the triangles $(\xi', 2\tilde \xi, \eta')$ and $(\xi', \xi, 0)$ 
are similar. Thus, by \inc{1} of $M_s$,
\[
N_s(\xi') = M_s(\xi') 
\ge
\tfrac1\nu M_s( \nu \xi' )
=
\tfrac1\nu M_s( \zeta' ) 
\ge
\tfrac1{2\nu} M_s(\tilde \xi)
\ge
\tfrac1{2\alpha'} M_s(\tilde \xi),
\]
where we used the conclusion of the previous paragraph in the penultimate step. 
The inequality 
\[
M_s\big(\tilde \xi\big)
\le 2\alpha' N_s(\xi')
\]
follows, so \eqref{eq:Cineq} holds with $C=2$. 

\textbf{Case 2: $\tilde \xi \not\in K_s$.}  
Let $\nu:= \|\tilde \xi \|_{K_s}^{-1} < 1$. Since $K_s$ is closed, it follows from the definition 
of $\|\cdot\|_{K_s}$ that $\Phi(\nu \tilde \xi) = s$ and $\nu \tilde \xi \in \partial K_s$.
Furthermore, $N_s(\nu\tilde\xi)=s=M_s(\nu\tilde\xi)$ and so 
\begin{align*}
M_s(\tilde \xi)
= s \|\tilde \xi\|_{K_s}
= \tfrac1\nu s 
&= 
\tfrac1\nu M_s(\nu \tilde \xi) 
\le
\tfrac4\nu (\alpha \Phi(\nu \xi) + \alpha' N_s(\nu \xi')) 
\le
4(\alpha \Phi(\xi) + \alpha' N_s(\xi')),
\end{align*}
where we used the previous case for $\nu \tilde \xi \in K_s$ in the 
first inequality and \inc{1} for the last step. 
%
\end{proof}

We are ready to prove the main theorem, i.e.\ the equivalence of 
\aone{} and \condM{}.

\begin{proof}[Proof of Theorem~\ref{thm:equiv}]
Since $(\Phi_{B}^-)^\conv\le \Phi_B^-$, it follows from \condM{} that
\[
\Phi_B^+(\beta \xi) 
\le 
(\Phi_B^-)^\conv (\xi)+1
\le
\Phi_B^-(\xi)+1,
\]
when $\xi\in \Rm$ with $\Phi_B^-(\xi)\le \frac K{\mu(B)}$, where $B\subset\Rn$ is a ball, 
which gives \aone{}. 

Assume now conversely that \aone{} holds and let $s:=\frac K{\mu(B)}+1$ for a ball $B\subset\Rn$ 
with $\mu(B)\le 1$. 
Define  $N_s$ as before based on $K_s:=\{\xi\in \Rm\,|\, \Phi_B^+(\beta\xi) \le s\}$ 
and set $M_s(\xi) := \min \{\Phi_B^+(\beta\xi), N_s(\xi)\}$.
By Proposition~\ref{prop:almostConvex}, $M_s$ is almost convex so 
$M_s(\beta'\xi) \le (M_s)^\conv(\xi)$ by Corollary~\ref{cor:Mstar}. 

If $\xi\in K_s$, then $M_s(\xi) = \Phi_B^+(\beta\xi)\le s$. 
Now either $\Phi_B^-(\xi)\le \frac K{\mu(B)}$ in which case \aone{} implies that 
$\Phi_B^+(\beta\xi) \le \Phi_{B}^-(\xi)+1$,
or $\Phi_B^-(\xi)> \frac K{\mu(B)}$ in which case $\Phi_B^+(\beta\xi)\le s \le \Phi_{B}^-(\xi)+1$.
Combining the two cases, we find that  
\[
\Phi_B^+(\beta\xi) \le \Phi_{B}^-(\xi)+1\qquad\text{for all } \xi\in K_s.
\]
If $\xi\not\in K_s$, then $\nu:= \| \xi \|_{K_s}^{-1} < 1$. 
As in Case 2 of the previous proof $\nu \xi \in \partial K_s$ and $\Phi_B^+(\beta\nu\xi)=s$.
If $\Phi_B^-(\nu \xi)< \frac K{\mu(B)}$, then \aone{} implies that 
$\Phi_B^+(\beta\nu\xi) \le \Phi_B^-(\nu \xi) + 1< s$, which is a contradiction. 
Therefore $\Phi_B^-(\nu \xi)\ge \frac K{\mu(B)} = s-1$ and so 
\begin{align*}
M_s(\xi)
= 
N_s(\xi)
=
\tfrac1\nu s
\le
\tfrac1\nu \tfrac s{s-1} \Phi_{B}^-(\nu\xi)
\le
\tfrac s{s-1}  \Phi_{B}^-(\xi),
\end{align*}
where we used \inc{1} of $\Phi_B^-$ in the last step. Note that $\frac s{s-1}=1+\frac{\mu(B)}K\le 1+\frac1K$ since we assumed that $\mu(B)\le 1$. 

In the previous paragraph we have shown that $M_s\le (1+\frac1K)\Phi_B^- + 1$. 
Therefore, the convex minorant of $M_s$ is also a convex minorant of $(1+\frac1K)\Phi_B^-+1$, and 
we conclude that $(M_s)^\conv\le (1+\tfrac1K)(\Phi_B^-)^\conv+1$
since $(\Phi_B^-)^\conv$ is the greatest convex minorant of $\Phi_B^-$. 
We noted above that $M_s(\beta'\xi) \le (M_s)^\conv(\xi)$. Therefore,
\[
M_s(\beta' \xi) \le (1+\tfrac1K)(\Phi_B^-)^\conv(\xi)+1
\qquad\text{for all }\xi\in \Rm. 
\]

Let us show that \condM{} holds. Assume that $(\Phi_B^-)^\conv(\xi)\le \frac K{\mu(B)}$. 
By \inc{1} and the conclusion of the previous paragraph, 
\[
M_s(\tfrac K{K+1}\beta' \xi) 
\le 
\tfrac K{K+1} M_s(\beta' \xi) 
\le 
(\Phi_B^-)^\conv(\xi)+1 \le s.
\]
Therefore $\tfrac K{K+1}\beta' \xi\in K_s$
and $M_s(\tfrac K{K+1}\beta' \xi) =\Phi_B^+(\tfrac K{K+1}\beta\beta' \xi)$. 
Thus 
\[
\Phi_B^+(\tfrac K{K+1}\beta\beta' \xi)\le (\Phi_B^-)^\conv(\xi)+1
\] 
and we have established \condM{} with constant $\frac K{K+1} \beta\beta'$.
\end{proof}
  
The assumption $\Phi_B^-(\xi)\le \frac K{\mu(B)}$ from \aone{} is 
somewhat difficult to verify. In the 
isotropic case, if we assume that $\rho_\phi(f)\le 1$, then we can conclude from 
Jensen's inequality that 
\begin{equation}\label{eq:A1condition}
\phi_B^-\bigg( \beta \fint_B f\, d\mu\bigg) \le \fint_B \phi(x,|f|)\, d\mu\le \frac1{\mu(B)}.
\end{equation}
Thus we may apply \aone{} to conclude that 
\[
\phi_B^+\bigg( \beta^2 \fint_B f\, d\mu\bigg) 
\le 
\phi_B^-\bigg( \beta \fint_B f\, d\mu\bigg) + 1.
\]
This argument is not possible in the anisotropic case, since $(\Phi_B^-)^\conv$ is not 
comparable to $\Phi_B^-$.
Fortunately, the condition of \condM{} is easier to use. 

\begin{proof}[Proof of Corollary~\ref{cor:jensen}]
By Theorem~\ref{thm:equiv}, $\Phi$ satisfies \condM{}. 
Since $(\Phi_B^-)^\conv$ is convex, it follows by Jensen's inequality that 
\[
(\Phi_B^-)^\conv \bigg(  \fint_B f\, d\mu\bigg) 
\le 
\fint_B (\Phi_B^-)^\conv (f)\, d\mu 
\le 
\fint_B \Phi(x,f)\, d\mu \le \frac 1{\mu(B)}.
\]
Therefore we can use \condM{} with $\xi = \fint_B f\, d\mu$ and the previous inequality 
to conclude that 
\[
\Phi_B^+ \bigg( \beta \fint_B f\, d\mu\bigg) 
\le 
(\Phi_B^-)^\conv \bigg( \fint_B f\, d\mu\bigg) + 1
\le 
\fint_B \Phi(x,f)\, d\mu + 1. \qedhere
\]
\end{proof}

\section*{Acknowledgment}

I would like to thank Iwona Chlebicka for her help in explaining 
\cite[Remarks 2.3.14 and 3.7.6]{ChlGSW21} and 
the Wihuri Foundation for financial support.


\bigskip

 \noindent\small{\textsc{P. Hästö}\\ 
Department of Mathematics and Statistics,
FI-20014 University of Turku, Finland}\\
\footnotesize{\texttt{peter.hasto@utu.fi}}\\

\end{document}